\documentclass[12pt]{article} 

\usepackage[utf8]{inputenc} 

\usepackage{geometry} 
\usepackage{graphicx} 
\usepackage[bb=boondox,bbscaled=.95,cal=boondoxo]{mathalfa}
\usepackage[all,cmtip]{xy}
\usepackage{booktabs} 
\usepackage{array} 
\usepackage{paralist} 
\usepackage{verbatim} 
\usepackage{subfig} 

\usepackage{amsfonts,amsmath,latexsym,amssymb} 
\usepackage{amsthm}            
\usepackage{mathptmx}

\usepackage{booktabs} 
\usepackage{array} 
\usepackage{paralist} 
\usepackage{verbatim} 
\usepackage{subfig} 
\usepackage{amsfonts,amsmath,latexsym,amssymb} 
\usepackage{amsthm}            
\usepackage{mathrsfs,upref}         
\usepackage{mathptmx}
\usepackage{fancyhdr} 
\usepackage[all,cmtip]{xy}
\usepackage{tikz-cd}
\usepackage{tikz}
\usetikzlibrary{trees}
\usepackage{sectsty}
\allsectionsfont{\sffamily\mdseries\upshape} 

\usepackage[nottoc,notlof,notlot]{tocbibind} 
\usepackage[titles,subfigure]{tocloft} 

\newtheorem{thm}{\textsc{Theorem}}
\newtheorem{cor}[thm]{\textsc{Corollary}}
\newtheorem{lem}[thm]{\textsc{Lemma}}

\newtheorem{prop}[thm]{\textsc{Proposition}}
\newtheorem*{prop*}{Proposition}
\newtheorem*{lem*}{\textsc{Lemma}}

\newtheorem*{fact*}{\textsc{\textbf{Fact}}}

\newtheorem*{exa*}{Example}
\theoremstyle{definition}
\newtheorem*{ack}{Acknowledgement}
\newtheorem{defn}[thm]{Definition}
\theoremstyle{remark}

\newtheorem{thm*}[thm]{Theorem}

\DeclareMathOperator{\im}{Im}

\DeclareMathOperator{\id}{Id}
\DeclareMathOperator{\h}{\mathrm{H}}
\DeclareMathOperator{\K}{\mathcal{K}}
\DeclareMathOperator{\fact}{Lemma~1}
\DeclareMathOperator{\proposition}{\textsc{Proposition}}

\DeclareMathOperator{\lemma}{\textsc{Lemma}}
\title{$q$-Independence of the $C^{*}$-algebra of  "Continuous Functions" on a Quantum $2\times 2$-Matrix Ball}
\author{Olof Giselsson}
\date{}
\begin{document}
\maketitle
\begin{abstract}
For $0<q<1,$ let $\mathrm{Pol}(\mathrm{Mat}_{2})_q$ be the $q$-analogue of regular functions on the open matrix ball of $2\times 2$-matrices introduced by L. Vaksman.
We show that the universal enveloping $C^{*}$-algebras of $\mathrm{Pol}(\mathrm{Mat}_{2})_{q},$ $0<q<1,$ are isomorphic for all values of $q.$ 
\end{abstract}
\section{Introduction}
Let $q\in (0,1).$ In~\cite{SV2}, L. Vaksman and S. Sinel’shchikov put forward a construction of a $q$-analogue of Hermitian symmetric spaces of non-compact type via a $q$-analogue of the Harish-Chandra embedding. This construction yields a $q$-analogue of the $*$-algebra of polynomial functions on the bounded symmetric domain that is the image of the Harish-Chandra embedding. In the simplest case (see~\cite{SV2}, section $9$), this yields the algebra $\mathrm{Pol}(\mathbb{C})_{q},$ the quantum disc. This is the $*$-algebra generated by a single generator $z$ subject to the relation
$$
z^{*} z=q^{2}z z^{*}+(1-q^{2})I.
$$
Another particular case, derived explicitly in~\cite{ssvs1}, gives the algebra $\mathrm{Pol}(\mathrm{Mat}_{n})_{q},$ a $q$-analogue of polynomial functions on the open matrix ball $\mathbb{D}_{n}:=\{Z\in \mathrm{Mat}_{n}:Z^{*}Z<I\}.$
\\

Let $C(\overline{\mathbb{D}})_{q}$ be the universal enveloping $C^{*}$-algebra of $\mathrm{Pol}(\mathbb{C}).$ It follows from the classification of its irreducible representations [~\cite{vak}, Corollary 1.13.], that for all $q,$ the $C(\overline{\mathbb{D}})_{q}$ is isomorphic to the Toeplitz algebra $C^{*}(S)$, i.e. the $C^{*}$-algebra generated by the unilateral shift $S\in\mathcal{B}(\ell^{2}(\mathbb{Z}_{+}))$ acting as $e_{k}\mapsto e_{k+1}$ on the natural orthonormal basis $\{e_{k}\}_{k\in \mathbb{Z}_{+}}$. In particular, it follows that the algebras $C(\overline{\mathbb{D}})_{q},$ $0<q<1,$ are all isomorphic.
\\

The aim of this paper is to prove a similar result for $\mathrm{Pol}(\mathrm{Mat}_{2})_{q}:$ The universal enveloping algebras $C(\overline{\mathbb{D}}_{2})_{q}$ of $\mathrm{Pol}(\mathrm{Mat}_{2})_{q}$ are isomorphic for all values of $q.$
This ties in with the larger project of trying to understand how, in $q$-analogues, the $C^{*}$-algebraic structure depends on the deformation parameter. The development up to now seems to indicate that they usually are independent of it, at least for $0<q<1.$ This was observed by S.L. Woronowicz~\cite{woron} for his compact quantum group $SU_{q}(2)$. Same results has been proven the Vaksman-Soibelman spheres, as well as similar objects by relating them to graph $C^*$-algebras. In the 90's G. Nagy proved it for $SU_{q}(3)$ (~\cite{gnagy}), whose proof was then extended by the author to work for all compact semi-simple Lie groups~\cite{gisel2}.
It also interesting to look at the limit $q\to 0,$ as for some of the examples listed above (e.g. $SU_{q}(2)$, Vaksman-Soibelman spheres), the $q$-independence can be proven by showing ismorphism with $q=0$, and that this is a graph-$C^*$-algebra ~\cite{woj}. A resent paper by M. Matassa and R. Yuncken~\cite{roma} generalized this and showed that in the case of the usual $q$-deformed compact Lie groups, one can use crystal bases to show that the limit $q\to 0$ $C^*$-algebra is coming from a higher rank graph. However, in this case (as of now), there is still the question of ismorphisms with non-zero values of $q.$
\begin{ack}
This work was supported by the RCN (Research Council of Norway), grant 300837.
\end{ack}

\subsection{ The Quantum group $\mathbb{C}[SU_{n}]_{q}$ and its irreducible representations}
Recall the definition of the Hopf algebra
$\mathbb{C}[SL_n]_q$. It is defined by the generators $\{t_{i,j}\}_{i,j=1,\ldots,n}$ and the relations
\begin{flalign*}
& t_{\alpha, a}t_{\beta, b}-qt_{\beta, b}t_{\alpha, a}=0, & a=b \quad \& \quad
\alpha<\beta,& \quad \text{or}\quad a<b \quad \& \quad \alpha=\beta,
\\ & t_{\alpha, a}t_{\beta, b}-t_{\beta, b}t_{\alpha, a}=0,& \alpha<\beta \quad
\&\quad a>b,&
\\ & t_{\alpha, a}t_{\beta, b}-t_{\beta, b}t_{\alpha, a}-(q-q^{-1})t_{\beta, a}
t_{\alpha, b}=0,& \alpha<\beta \quad \& \quad a<b, &
\\ & \det \nolimits_q \mathbf{t}=1.
\end{flalign*}
Here $\det_q \mathbf{t}$ is the  q-determinant of the matrix
$\mathbf{t}=(t_{i,j})_{i,j=1}^{n}$.

It is well known (see \cite{KlSh} or other standard book on quantum groups) that
$\mathbb{C}[SU_n]_q\stackrel{\mathrm{def}}{=}(\mathbb{C}[SL_n]_q,\star)$ is
a Hopf $*$-algebra. The co-multiplication $\Delta$, the co-unit $\varepsilon$, the antipode $S$ and the involution $\star$ are defined as follows
\begin{equation*}
\Delta(t_{i,j})=\sum_k t_{i,k}\otimes t_{k,j},\quad \varepsilon(t_{i,j})=\delta_{ij},\quad S(t_{i,j})=(-q)^{i-j}\det\nolimits_q\mathbf{t}_{ji},
\end{equation*}
 and
\begin{equation*}\label{star5}
t_{i,j}^\star=(-q)^{j-i}\det \nolimits_q\mathbf{t}_{ij},
\end{equation*}
where ${\mathbf t}_{ij}$ is the matrix derived from ${\mathbf t}$ by discarding its $i$-th row and $j$-th column.
\\

Let $\{e_{k}|k\in \mathbb{Z}_{+}\}$ be the standard orthonormal basis of $\ell^{2}(\mathbb{Z}_{+})$ and let $S$ be the isometric shift $e_{k}\mapsto e_{k+1}$ on $\ell^{2}(\mathbb{Z}_{+}).$ For $q\in (0,1),$ we now introduce the operators $d_{q},C_{q}\in \mathcal B(\ell^{2}(\mathbb{Z}_{+}))$ given by the formulas 
\begin{equation}\label{Cd}
\begin{array}{ccc}
C_{q}e_{m}=\sqrt{1-q^{2m}}e_{m}, &d_{q}e_{m}=q^{m}e_{m}& \text{ for all $k\in \mathbb{Z}_{+}.$}\\
\end{array} 
\end{equation}
Now let 

\begin{equation}\label{fan}
\begin{array}{cc}
T_{11}=S^{*}C_{q}, & T_{12}=-q d_{q},\\
T_{21}=d_{q}, & T_{22}=C_{q}S\\.
\end{array}
\end{equation}
It is not hard to verify the formulas 
$$
\begin{array}{cc}d_{q}=\sum_{j=0}^{\infty}q^{j}S^{j}(I-S S^{*})S^{*j}, & C_{q}^{2}=I-d_{q}^{2}\end{array}.
$$
It follows from this that $d_{q},C_{q}\in C^{*}(S),$ the $C^{*}$-algebra generated by $S,$ and therefore $T_{ij}\in C^{*}(S),$ for $i,j=1,2.$ 
\\

Notice that the following relations hold
\begin{equation}\label{tsu2q}
\begin{array}{ccc}
T_{11}=T_{22}^{*}& T_{11}T_{22}-q T_{12}T_{21}=0 & T_{12}T_{21}=T_{21}T_{12}\\
T_{11}T_{12}=qT_{12}T_{11} & T_{11}T_{21}=qT_{21}T_{11}& T_{21}=\sqrt{I-T_{22}T_{11}}\\
& T_{11}T_{22}=q^{2}T_{22}T_{11}+(1-q^{2})I &\\
\end{array}
\end{equation}
from which we conclude that the operators $T_{ij}$ determine an irreducible $*$-representation $\pi:\mathbb{C}[SU_{2}]_{q}\to C^{*}(S)\subset \mathcal B(\ell^{2}(\mathbb{Z}_{+}))$  \begin{equation}\label{basrep}\begin{array}{cccc}\pi(t_{ij})= T_{ij}, & \text{(for $1\leq i,j\leq 2$)}.\end{array}\end{equation}
For $1\leq i\leq n-1,$ let $\phi_{i}:\mathbb{C}[SU_{n}]_{q}\rightarrow \mathbb{C}[SU_{2}]_{q}$ be the $*$-homomorphism determined by

\begin{equation}\label{phi}
\begin{array}{ccc}
\phi_{i}(t_{i,i})=t_{1, 1}, & & \phi_{i}(t_{i+1,i+1})=t_{2, 2},\\
\phi_{i}(t_{i,i+1})=t_{1, 2}, & & \phi_{i}(t_{i+1,i})=t_{2, 1}\\
 &\text{and } \phi_{i}(t_{k,j})=\delta_{k,j}I \text{ otherwise.}&
\end{array}
\end{equation}
Then $\pi_{i}=\pi\circ \phi_{i}$ is a $*$-representation of $\mathbb{C}[SU_{n}]_{q}$ on $\ell^{2}(\mathbb{Z}_{+}).$ Moreover $\pi_{i}(\mathbb{C}[SU_{n}]_{q})\subseteq C^{*}(S).$
\\

Let $s_{i}$ denote the adjacent transposition $(i,i+1)$ in the symmetric group $S_{n}.$ 
\begin{defn}\label{surep1}
For an element $s\in S_{n}$ consider a reduced decomposition of $s=s_{j_{1}}s_{j_{2}}\dots s_{j_{m}}$ into a product of adjacent transposition and put
$$
\pi_{s}=\pi_{j_{1}}\otimes \dots \otimes \pi_{j_{m}}.
$$
\end{defn}
It is known that $\pi_{s}$ is independent of the specific reduced expression of $s,$ in the sense that another reduced decomposition gives a unitary equivalent $*$-representation.
\\

Recall that the \textit{length} of $s\in S_{n},$ denoted by $\ell(s),$ is the number of adjacent transpositions in a reduced decomposition of $s=s_{j_{1}}s_{j_{2}}\dots s_{j_{\ell(s)}}.$ For the identity element $e\in S_{n}$, we let $\ell(e)=0.$ For $s,t\in S_{n},$ it is easy to see that the inequality
$$
\ell(st)\leq \ell(s)+\ell(t).
$$
holds, and that $\ell(s^{-1})=\ell(s).$
\\

Let $\varphi=\{\varphi_{1},\dots,\varphi_{n}\}\in [0,2\pi)^{n}$ be a $n$-tuple such that 
$$
\sum_{j=1}^{n}\varphi_{j}\equiv 0 \pmod{2\pi}.
$$
Then we can define a one-dimensional $*$-representation $\chi_{\varphi}:\mathbb{C}[SU_{n}]_{q}\to \mathbb{C}$ by the formula
\begin{equation}\label{chi}
\chi_{\varphi}(t_{ij})=e^{i\varphi_{j}}\delta_{ij}.
\end{equation}
From the work of Soibelman, the following is known (with $S_{n}$ the Weyl group of $SU_{n}$)
\begin{prop}\label{soib}[~\cite{KorS}, Theorem $6.2.7$]
Every irreducible $*$-representation $\Pi$ of $\mathbb{C}[SU_{n}]_{q}$ is equivalent to one of the form 
\begin{equation}\label{surep}
\pi_{s}\otimes \chi_{\varphi}
\end{equation}
for $s\in S_{n}$ and $\varphi=[\varphi_{1},\dots,\varphi_{n}]\in [0,2\pi)^{n},\sum_{j=1}^{n}\varphi_{j}\equiv 0 \pmod{2\pi}.$ Conversly, such pairs give rise to non-equivalent irreducible $*$-representations of $\mathbb{C}[SU_{n}]_{q}.$ 
\end{prop}
\subsection{The $*$-algebra $\mathrm{Pol}(\mathrm{Mat}_{n})_{q}$}
In what follows $\mathbb{C}$ is a ground field and $q \in(0,1)$. We assume
that all the algebras under consideration has a unit $I.$ Consider the well-known algebra
$\mathbb{C}[\mathrm{Mat}_n]_q$ defined by its generators $z_a^\alpha$,
$\alpha,a=1,\dots,n$, and the commutation relations
\begin{flalign}
& z_a^\alpha z_b^\beta-qz_b^\beta z_a^\alpha=0, & a=b \quad \& \quad
\alpha<\beta,& \quad \text{or}\quad a<b \quad \& \quad \alpha=\beta,
\label{zaa1}
\\ & z_a^\alpha z_b^\beta-z_b^\beta z_a^\alpha=0,& \alpha<\beta \quad
\&\quad a>b,& \label{zaa2}
\\ & z_a^\alpha z_b^\beta-z_b^\beta z_a^\alpha-(q-q^{-1})z_a^\beta
z_b^\alpha=0,& \alpha<\beta \quad \& \quad a<b. & \label{zaa3}
\end{flalign}
This algebra is a quantum analogue of the polynomial algebra
$\mathbb{C}[\mathrm{Mat}_n]$ on the space of $n \times n$-matrices.
Let $A=(a_{j,k})_{j,k}\in M_{n}(\mathbb{Z}_{+})$ be a $n\times n$ matrix of positive integers $a_{j,k}$, and denote by $z(A)$ the monomial
\begin{equation}\label{fock}
\begin{array}{cc}
z(A):=\\
(z_n^n)^{a_{n,n}}(z_n^{n-1})^{a_{n,n-1}}\ldots(z_n^1)^{a_{n,1}}(z_{n-1}^n)^{a_{n-1, n}}\ldots(z_{n-1}^1)^{a_{n-1, 1}}\ldots(z_1^n)^{a_{1,n}}\ldots (z_1^1)^{a_{1,1}}.
\end{array}
\end{equation}

 It follows from the Bergman diamond lemma that monomials $z(A)$, $A\in M_{n}(\mathbb{Z}_{+}),$ form a basis for the vector space ${\mathbb C}[\mathrm{Mat}_n]_q.$ Hence ${\mathbb C}[\mathrm{Mat}_n]_q$ admits a natural gradation given by $\text{deg} (z_{k}^{j})=1$, and in general 
$\text{deg} (z(A))=|A|,$ where
\begin{equation}\label{matrixnorm}
|A|=\sum_{k,j=1}^{n}a_{k,j}.
\end{equation}

In a similar way, we write $\mathbb{C}[\overline{\mathrm{Mat}}_n]_q$ for the algebra defined by generators
$(z_a^\alpha)^*$, $\alpha,a=1,\ldots,n$, and the relations
\begin{flalign}
& (z_b^\beta)^*(z_a^\alpha)^* -q(z_a^\alpha)^*(z_b^\beta)^*=0, \quad a=b
\quad \& \quad \alpha<\beta, \qquad \text{or} & a<b \quad \& \quad
\alpha=\beta, \label{zaa1*}
\\ & (z_b^\beta)^*(z_a^\alpha)^*-(z_a^\alpha)^*(z_b^\beta)^*=0,&
\alpha<\beta \quad\&\quad a>b,& \label{zaa2*}
\\ & (z_b^\beta)^*(z_a^\alpha)^*-(z_a^\alpha)^*(z_b^\beta)^*-
(q-q^{-1})(z_b^\alpha)^*(z_a^\beta)^*=0,& \alpha<\beta \quad \& \quad a<b. &
\label{zaa3*}
\end{flalign}
A gradation in $\mathbb{C}[\overline{\mathrm{Mat}}_n]_q$ is given by $\text{deg}(z_a^\alpha)^*=-1$.
\\

Consider now the algebra $\mathrm{Pol}(\mathrm{Mat}_n)_q,$ whose generators are $z_a^\alpha$, $(z_a^\alpha)^*$, $\alpha,a=1,\dots,n$, and the list of relations is formed by \eqref{zaa1} --
\eqref{zaa3*} and
\begin{flalign}
&(z_b^\beta)^*z_a^\alpha=q^2 \cdot \sum_{a',b'=1}^n
\sum_{\alpha',\beta'=1}^n R_{ba}^{b'a'}R_{\beta \alpha}^{\beta'\alpha'}\cdot
z_{a'}^{\alpha'}(z_{b'}^{\beta'})^*+(1-q^2)\delta_{ab}\delta^{\alpha \beta},
& \label{zaa4}
\end{flalign}
with $\delta_{ab}$, $\delta^{\alpha \beta}$ being the Kronecker symbols, and
$$
R_{ij}^{kl}=
\begin{cases}
q^{-1},& i \ne j \quad \& \quad i=k \quad \& \quad j=l
\\ 1,& i=j=k=l
\\ -(q^{-2}-1), & i=j \quad \& \quad k=l \quad \& \quad l>j
\\ 0,& \text{otherwise}.
\end{cases}
$$
The involution in $\mathrm{Pol}(\mathrm{Mat}_n)_q$ is introduced in an
obvious way: $*:z_a^\alpha \mapsto(z_a^\alpha)^*$.
\\

It is known from~\cite{ssv} (Corollary 10.4) that we have an isomorphism of vector spaces
\begin{equation}\label{isotensor}
\mathbb{C}[\mathrm{Mat}_n]_q\otimes\mathbb{C}[\overline{\mathrm{Mat}}_n]_q\rightarrow \mathrm{Pol}(\mathrm{Mat}_{n})_{q}
\end{equation}
induced by the linear map $a\otimes b\mapsto a\cdot b.$
Thus, we can consider the algebras $\mathbb{C}[\mathrm{Mat}_n]_q$ and $\mathbb{C}[\overline{\mathrm{Mat}}_n]_q$ as sub-algebras of $\mathrm{Pol}(\mathrm{Mat}_{n})_{q},$ if we identify them with the images of $\mathbb{C}[\mathrm{Mat}_n]_q\otimes I$ and $I\otimes\mathbb{C}[\overline{\mathrm{Mat}}_n]_q$ respectively.
\subsection{The Case $n=2$}
In the case of $n=2$, the relations defining $\mathrm{Pol}(\mathrm{Mat}_{2})_{q}$ specialized to the following:
$$
\begin{array}{cccc}
z_{1}^{1}z_{2}^{1}=qz_{2}^{1}z_{1}^{1}, & z_{2}^{1}z_{1}^{2}=z_{1}^{2}z_{2}^{1},\\
z_{1}^{1}z_{1}^{2}=q z_{1}^{2}z_{1}^{1}, &z_{2}^{1}z_{2}^{2}=q z_{2}^{2}z_{2}^{1},\\
z_{1}^{1}z_{2}^{2}-z_{2}^{2}z_{1}^{1}=(q-q^{-1})z_{1}^{2}z_{2}^{1}, & z_{1}^{2}z_{2}^{2}=qz_{2}^{2}z_{1}^{2},
\end{array}
$$
$$
\begin{array}{cccc}
(z_{1}^{1})^{*}z_{1}^{1}=q^{2}z_{1}^{1}(z_{1}^{1})^{*}-(1-q^{2})(z_{2}^{1}(z_{2}^{1})^{*}+z_{1}^{2}(z_{1}^{2})^{*})+\\
+q^{-2}(1-q^{2})^{2}z_{2}^{2}(z_{2}^{2})^{*}+1-q^{2},\\
(z_{2}^{1})^{*}z_{2}^{1}=q^{2}z_{2}^{1}(z_{2}^{1})^{*}-(1-q^{2})z_{2}^{2}(z_{2}^{2})^{*}+(1-q^{2}),\\
(z_{1}^{2})^{*}z_{1}^{2}=q^{2}z_{1}^{2}(z_{1}^{2})^{*}-(1-q^{2})z_{2}^{2}(z_{2}^{2})^{*}+(1-q^{2}),\\
(z_{2}^{2})^{*}z_{2}^{2}=q^{2}z_{2}^{2}(z_{2}^{2})^{*}+(1-q^{2})
\end{array}
$$
$$
\begin{array}{cccc}
(z_{1}^{1})^{*}z_{2}^{1}-qz_{2}^{1}(z_{1}^{1})^{*}=(q-q^{-1})z_{2}^{2}(z_{1}^{2})^{*}, &(z_{2}^{2})^{*}z_{2}^{1}=q z_{2}^{1}(z_{2}^{2})^{*},\\
(z_{1}^{1})^{*}z_{1}^{2}-qz_{1}^{2}(z_{1}^{1})^{*}=(q-q^{-1})z_{2}^{2}(z_{2}^{1})^{*}, &(z_{2}^{2})^{*}z_{1}^{2}=q z_{1}^{2}(z_{2}^{2})^{*},\\
(z_{1}^{1})^{*}z_{2}^{2}=z_{2}^{2}(z_{1}^{1})^{*}, &( z_{1}^{2})^{*}z_{2}^{1}=z_{2}^{1}(z_{1}^{2})^{*}
\end{array}
$$
From these relations it is clear that for $\varphi_{1},\varphi_{2}\in[0,2\pi),$ the map $z_{j}^{k}\mapsto e^{i \varphi_{j}}z_{j}^{k},$ extends to an automorphism $\Psi_{\varphi_{1},\varphi_{2}}:\mathrm{Pol}(\mathrm{Mat}_{2})_{q}\to\mathrm{Pol}(\mathrm{Mat}_{2})_{q}.$
\subsection{Irreducible representations of $\mathrm{Pol}(\mathrm{Mat}_{2})_{q}$}
There exists a special faitful irreducible representation of $\mathrm{Pol}(\mathrm{Mat}_{n})_{q}$ called the \textit{Fock representation}. This is the representation $\pi_{F}:\mathrm{Pol}(\mathrm{Mat}_{n})_{q}\to \mathcal{B}(\h_{F})$ determined by the existance of a cyclic vector $v_{0}\in \h_{F}$, called a \textit{vacuum vector}, such that 
\begin{equation}\label{fock}
\begin{array}{cccc}
\pi_{F}(z_{j}^{k})^{*}v_{0}=0, & \text{(for all $1\leq j,k\leq n$).}
\end{array}
\end{equation}
The Fock representation was introduced in~\cite{ssv}, where they also proved its existance, and moreover, that any other irreducible representation on a Hilbert space that contains a cyclic vector such that~\eqref{fock} holds, is unitarily equivalent to $\pi_{F}$. In~\cite{gisel1}, it was proven that the closure of $\im\,\pi_{F}\subseteq \mathcal{B}(\h_{F})$ is isomorphic to the universal enveloping $C^{*}$-algebra of $\mathrm{Pol}(\mathrm{Mat}_{n})_{q}.$ 
\\

In~\cite{bgt} we constructed the Fock representation the following way (specialized to the case $n=2$):
\begin{enumerate}[(i)]
\item
There exists a homomorphism $\zeta :\mathrm{Pol}(\mathrm{Mat}_{2})_{q}\to \mathbb{C}[SU_{4}]$ determined on the generators as
\begin{equation}\label{zeta}
\zeta(z_{k}^{j})=(-q)^{k-n}t_{n+k,n+j}.
\end{equation}
\item
Let $\sigma\in S_{4}$ (symmetric group of degree four) be the permutation
\begin{equation}\label{sigma}
\sigma=\left(\begin{array}{cccc}1&2&3&4\\ 3&4&1&2\end{array}\right).
\end{equation}
We can write $\sigma=s_{2}s_{1}s_{3}s_{2}$ and this is a reduced expression, so that $\ell(\sigma)=4$. Let $\pi_{\sigma}:\mathbb{C}[SU_{4}]_{q}\to \ell^{2}(\mathbb{Z}_{+})^{\otimes 4}$ be the representation associated to $\sigma$ (in the sense of Definition~\ref{surep1}).
\item Then the representation 

\begin{equation}\label{fock}\pi_{\sigma}\circ \zeta:\mathrm{Pol}(\mathrm{Mat}_{2})_{q}\to \mathcal{B}(\ell^{2}(\mathbb{Z}_{+})^{\otimes 4})\end{equation}
is irreducible and unitarily equivalent to the Fock representation, with a vacuum vector given by $$e_{0}\otimes e_{0}\otimes e_{0}\otimes e_{0}.$$ We can thus let $\pi_{F}$ to be defined by the formula~\eqref{fock}.
\end{enumerate}
In~\cite{gisel1}, we proved that~\eqref{fock} is part of a general pattern, in the following sense:
For every irreducible representation $\pi:\mathrm{Pol}(\mathrm{Mat}_{n})_{q}\to \mathcal{B}(\mathrm{K})$, there exists an irreducible representation $\Pi:\mathbb{C}[SU_{2n}]\to \mathcal{B}(\mathrm{K})$ such that the following diagram is commutative
$$
\begin{xy}\xymatrix{
\mathrm{Pol}(\mathrm{Mat}_{n})_{q}\ar[r]^*{\zeta}\ar[dr]_{\pi}& \mathbb{C}[SU_{2n}]_{q}\ar[d]^*{\Pi}\\
& \mathcal{B}(\mathrm{K})
}\end{xy}
$$
If we have an equivalence $\Pi\cong \pi_{\omega}\otimes \chi_{\varphi},$ then $\omega\in S_{2n}$ is uniquely determined, while $\varphi$ may be not (see~\cite{gisel1} Theorem 3.6).
\\

Recall that for $\varphi\in [0,2\pi),$ the coherent representation $\Omega_{\varphi}:\mathrm{Pol}(\mathrm{Mat}_{n})_{q}\to \mathcal{B}(\h)$ is determined by the existence of a non-zero cyclic vector $v\in \h$ such that $\Omega_{\varphi}(z_{1}^{1})v=e^{i\varphi}v,$ $\Omega_{\varphi}((z_{1}^{1})^{*})v=e^{-i \varphi}v$ and $\Omega_{\varphi}((z_{j}^{k})^{*})v=0$ for $(j,k)\neq(1,1).$ These representations are, up to equivalence, uniquely determined and non-equivalent for different $\varphi$ (Proposition 1.3.3 in~\cite{jsw}). Note that it follows from~\cite{gisel1} that we may assume $\h=\ell^{2}(\mathbb{Z}_{+})^{\otimes 3}$ and $v=e_{0}\otimes e_{0}\otimes e_{0}.$
\\

There are $7$ families of irreducible representations of $\mathrm{Pol}(\mathrm{Mat}_{2})_{q},$ and using the box-diagram presentations of irreducible representations of $\mathrm{Pol}(\mathrm{Mat}_{n})_{q}$ developed in~\cite{gisel1}, we can present them visually as
\begin{equation}\label{reps}
\begin{array}{c}
\begin{array}{ccccc}
\begin{tikzpicture}[thick,scale=0.5]
\draw (1,0) -- (0,0) -- (0,1)--(1,1)--(1,0)--(0,0);
\draw (1,1) -- (0,1) -- (0,2)--(1,2)--(1,1)--(0,1);
\draw (2,0) -- (1,0) -- (1,1)--(2,1)--(2,0)--(1,0);
\draw (2,1) -- (1,1) -- (1,2)--(2,2)--(2,1)--(1,1);
\node at (0.5,-0.5) {$1$};
\node at (1.5,-0.5) {$2$};
\node at (2.5,0.5) {$2$};
\node at (2.5,1.5) {$1$};
\end{tikzpicture}
&
\begin{tikzpicture}[thick,scale=0.5]
\draw (1,0) -- (0,0) -- (0,1)--(1,1)--(1,0)--(0,0);
\filldraw[fill=black!20!white, draw=black]  (1,1) -- (0,1) -- (0,2)--(1,2)--(1,1)--(0,1);
\draw (2,0) -- (1,0) -- (1,1)--(2,1)--(2,0)--(1,0);
\draw (2,1) -- (1,1) -- (1,2)--(2,2)--(2,1)--(1,1);
\node at (0.5,-0.5) {$1$};
\node at (1.5,-0.5) {$2$};
\node at (2.5,0.5) {$2$};
\node at (2.5,1.5) {$1$};
\end{tikzpicture}
&
\begin{tikzpicture}[thick,scale=0.5]
\filldraw[fill=black!20!white, draw=black](1,0) -- (0,0) -- (0,1)--(1,1)--(1,0)--(0,0);
\filldraw[fill=black!50!white, draw=black]  (1,1) -- (0,1) -- (0,2)--(1,2)--(1,1)--(0,1);
\draw (2,0) -- (1,0) -- (1,1)--(2,1)--(2,0)--(1,0);
\draw (2,1) -- (1,1) -- (1,2)--(2,2)--(2,1)--(1,1);
\node at (0.5,-0.5) {$1$};
\node at (1.5,-0.5) {$2$};
\node at (2.5,0.5) {$2$};
\node at (2.5,1.5) {$1$};
\end{tikzpicture}
&
\begin{tikzpicture}[thick,scale=0.5]
\draw (1,0) -- (0,0) -- (0,1)--(1,1)--(1,0)--(0,0);
\filldraw[fill=black!50!white, draw=black] (1,1) -- (0,1) -- (0,2)--(1,2)--(1,1)--(0,1);
\draw(2,0) -- (1,0) -- (1,1)--(2,1)--(2,0)--(1,0);
\filldraw[fill=black!20!white, draw=black] (2,1) -- (1,1) -- (1,2)--(2,2)--(2,1)--(1,1);
\node at (0.5,-0.5) {$1$};
\node at (1.5,-0.5) {$2$};
\node at (2.5,0.5) {$2$};
\node at (2.5,1.5) {$1$};
\end{tikzpicture}
&
\begin{tikzpicture}[thick,scale=0.5]
\filldraw[fill=black!20!white, draw=black] (1,0) -- (0,0) -- (0,1)--(1,1)--(1,0)--(0,0);
\filldraw[fill=black!50!white, draw=black] (1,1) -- (0,1) -- (0,2)--(1,2)--(1,1)--(0,1);
\draw (2,0) -- (1,0) -- (1,1)--(2,1)--(2,0)--(1,0);
\filldraw[fill=black!20!white, draw=black](2,1) -- (1,1) -- (1,2)--(2,2)--(2,1)--(1,1);
\node at (0.5,-0.5) {$1$};
\node at (1.5,-0.5) {$2$};
\node at (2.5,0.5) {$2$};
\node at (2.5,1.5) {$1$};
\end{tikzpicture}
\end{array}
\\
\begin{array}{ccc}
\begin{tikzpicture}[thick,scale=0.5]
\filldraw[fill=black!50!white, draw=black] (1,0) -- (0,0) -- (0,1)--(1,1)--(1,0)--(0,0);
\filldraw[fill=black!50!white, draw=black] (1,1) -- (0,1) -- (0,2)--(1,2)--(1,1)--(0,1);
\filldraw[fill=black!20!white, draw=black] (2,0) -- (1,0) -- (1,1)--(2,1)--(2,0)--(1,0);
\draw(2,1) -- (1,1) -- (1,2)--(2,2)--(2,1)--(1,1);
\node at (0.5,-0.5) {$1$};
\node at (1.5,-0.5) {$2$};
\node at (2.5,0.5) {$2$};
\node at (2.5,1.5) {$1$};
\end{tikzpicture}
&
\begin{tikzpicture}[thick,scale=0.5]
\filldraw[fill=black!50!white, draw=black] (1,0) -- (0,0) -- (0,1)--(1,1)--(1,0)--(0,0);
\filldraw[fill=black!50!white, draw=black] (1,1) -- (0,1) -- (0,2)--(1,2)--(1,1)--(0,1);
\filldraw[fill=black!20!white, draw=black] (2,0) -- (1,0) -- (1,1)--(2,1)--(2,0)--(1,0);
\filldraw[fill=black!20!white, draw=black](2,1) -- (1,1) -- (1,2)--(2,2)--(2,1)--(1,1);
\node at (0.5,-0.5) {$1$};
\node at (1.5,-0.5) {$2$};
\node at (2.5,0.5) {$2$};
\node at (2.5,1.5) {$1$};
\end{tikzpicture}
\end{array}

\end{array}
\end{equation}

Let us recall how to interpret the diagrams in~\eqref{reps}. The smaller boxes corresponds to the tensor factors in $\pi_{F}=\pi_{\sigma}\circ\zeta=(\pi_{2}\otimes\pi_{1}\otimes\pi_{3}\otimes \pi_{2})\circ\zeta$ (where $\sigma\in S_{4}$ is the permutation~\eqref{sigma}) as follows:
$$
\begin{array}{ccc}
\begin{tikzpicture}[baseline=12,thick,scale=0.5]
\draw (1,0) -- (0,0) -- (0,1)--(1,1)--(1,0)--(0,0);
\draw (1,1) -- (0,1) -- (0,2)--(1,2)--(1,1)--(0,1);
\draw (2,0) -- (1,0) -- (1,1)--(2,1)--(2,0)--(1,0);
\draw (2,1) -- (1,1) -- (1,2)--(2,2)--(2,1)--(1,1);
\node at (0.5,-0.5) {$1$};
\node at (1.5,-0.5) {$2$};
\node at (2.5,0.5) {$2$};
\node at (2.5,1.5) {$1$};
\node at (0.5,0.5) {$1$};
\node at (0.5,1.5) {$2$};
\node at (1.5,0.5) {$3$};
\node at (1.5,1.5) {$4$};
\end{tikzpicture}
& \Longleftrightarrow &(\underset{1}{\pi_{2}}\otimes\underset{2}{\pi_{1}}\otimes\underset{3}{\pi_{3}}\otimes \underset{4}{\pi_{2}})\circ \zeta
\end{array}
$$ 
For $\varphi\in [0,2\pi),$ let $\tau_{\varphi}:C^{*}(S)\to \mathbb{C}$ be the composition of homomorphism $C^{*}(S)\to C(\mathbb{T})\to \mathbb{C}$, where the first arrow is the homomorphism extending $S\mapsto z$ (here $z\in C(\mathbb{T})$ is the coordinate function), and the second arrow is evaluation at $e^{i\varphi}$ (note that $\tau_{\varphi}(T_{12})=\tau_{\varphi}(T_{21})=0,$ $\tau_{\varphi}(T_{11})=e^{-i\varphi}$ and $\tau_{\varphi}(T_{22})=e^{i\varphi}$). The coloring of the box-diagrams is to be interpreted as follows:
\begin{itemize}
\item
If a box is dark gray, we apply $\tau_{0}$ onto the corresponding tensor factor,
\item
if a box is light gray, we apply $\tau_{\varphi}$ onto the corresponding tensor factor, 
\item
if the box is white, we do nothing to the corresponding factor. 
\end{itemize}
In this way every box diagram clearly corresponds to a family of representations of $\mathrm{Pol}(\mathrm{Mat}_{2})_{q}$, and it is a result from~\cite{gisel1} that there is a one to one correspondence between elements in these families and equivalence classes of irreducible representations of $\mathrm{Pol}(\mathrm{Mat}_{2})_{q}$ (see also the lists found in~\cite{turow} or~\cite{pro_tur}). In particular, the all-white box diagram is the Fock representation, and the box diagram with a light gray box in the upper left corner corresponds to the family of coherent representations. One of the main features of the box diagrams presentation is that they give a simple way of calculating the images of the generators under the representation using directed connected paths of hooks and arrows. We refer to Section 3.1 in~\cite{gisel1} for a detailed exposition. To make it easier for the reader, we write out the images of the generators under the Fock representation
\begin{equation}\label{fock}
\begin{array}{ccc}
\pi_{F}(z_{1}^{1})=d_{q}\otimes S^{*}C_{q}\otimes I\otimes d_{q}+(-q)^{-1}C_{q}S\otimes I\otimes S^{*}C_{q}\otimes C_{q}S, \\ \pi_{F}(z_{1}^{2})=C_{q}S\otimes I\otimes d_{q}\otimes I,\\
\pi_{F}(z_{2}^{1})=I\otimes I\otimes d_{q}\otimes C_{q}S, \\ \pi_{F}(z_{2}^{2})=I\otimes I \otimes C_{q}S\otimes I.
\end{array}
\end{equation}
To calculate the image of other representation, one simply apply $\tau_{\varphi}$ or $\tau_{0}$ onto the tensor factors according to the corresponding box-diagram,
\\

By taking quotients of tensor factors in~\eqref{reps}, and using appropriate values of $\varphi$ on the light gray boxes, one can present the inclusions of kernels of irreducible representations by the following diagram (here, an arrow $\pi\to \pi'$ means $\ker \pi'\subseteq \ker \pi$)
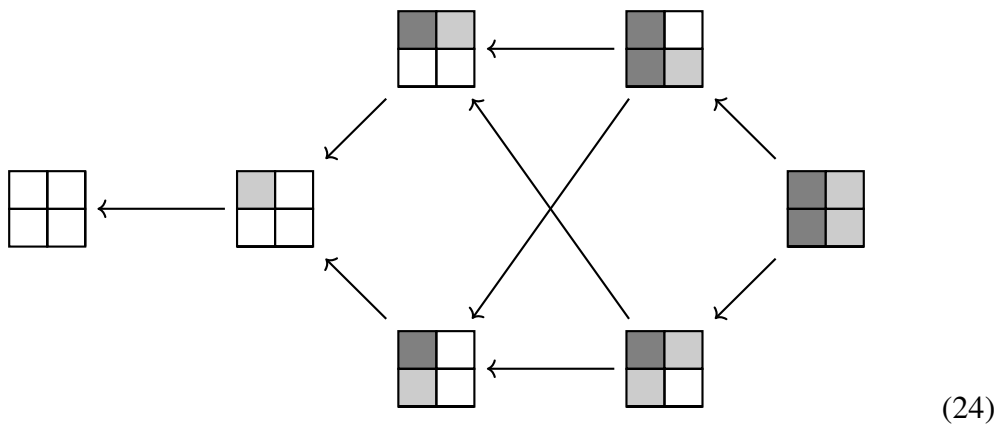
\begin{equation}\label{jacob}
\begin{tikzpicture}[node distance=3cm,rotate=90,transform shape]
 \node (Qf)                  {$\begin{tikzpicture}[thick,scale=0.5]
\draw (1,0) -- (0,0) -- (0,1)--(1,1)--(1,0)--(0,0);
\draw (1,1) -- (0,1) -- (0,2)--(1,2)--(1,1)--(0,1);
\draw (2,0) -- (1,0) -- (1,1)--(2,1)--(2,0)--(1,0);
\draw (2,1) -- (1,1) -- (1,2)--(2,2)--(2,1)--(1,1);
\end{tikzpicture}$};
\node (Qc) [below of=Qf] {\begin{tikzpicture}[thick,scale=0.5,rotate=-90,transform shape]
\draw (1,0) -- (0,0) -- (0,1)--(1,1)--(1,0)--(0,0);
\filldraw[fill=black!20!white, draw=black]  (1,1) -- (0,1) -- (0,2)--(1,2)--(1,1)--(0,1);
\draw (2,0) -- (1,0) -- (1,1)--(2,1)--(2,0)--(1,0);
\draw (2,1) -- (1,1) -- (1,2)--(2,2)--(2,1)--(1,1);
\end{tikzpicture}};
\node (Q12) [below left of=Qc] {\begin{tikzpicture}[thick,scale=0.5,rotate=-90,transform shape]
\filldraw[fill=black!20!white, draw=black](1,0) -- (0,0) -- (0,1)--(1,1)--(1,0)--(0,0);
\filldraw[fill=black!50!white, draw=black]  (1,1) -- (0,1) -- (0,2)--(1,2)--(1,1)--(0,1);
\draw (2,0) -- (1,0) -- (1,1)--(2,1)--(2,0)--(1,0);
\draw (2,1) -- (1,1) -- (1,2)--(2,2)--(2,1)--(1,1);
\end{tikzpicture}};
\node (Q21) [below right of=Qc] {\begin{tikzpicture}[thick,scale=0.5,rotate=-90,transform shape]
\draw (1,0) -- (0,0) -- (0,1)--(1,1)--(1,0)--(0,0);
\filldraw[fill=black!50!white, draw=black] (1,1) -- (0,1) -- (0,2)--(1,2)--(1,1)--(0,1);
\draw(2,0) -- (1,0) -- (1,1)--(2,1)--(2,0)--(1,0);
\filldraw[fill=black!20!white, draw=black] (2,1) -- (1,1) -- (1,2)--(2,2)--(2,1)--(1,1);
\end{tikzpicture}};
\node (Q10) [below of=Q12] {\begin{tikzpicture}[thick,scale=0.5,rotate=-90,transform shape]
\filldraw[fill=black!20!white, draw=black] (1,0) -- (0,0) -- (0,1)--(1,1)--(1,0)--(0,0);
\filldraw[fill=black!50!white, draw=black] (1,1) -- (0,1) -- (0,2)--(1,2)--(1,1)--(0,1);
\draw (2,0) -- (1,0) -- (1,1)--(2,1)--(2,0)--(1,0);
\filldraw[fill=black!20!white, draw=black](2,1) -- (1,1) -- (1,2)--(2,2)--(2,1)--(1,1);
\end{tikzpicture}};
\node (Q01) [below of=Q21] {\begin{tikzpicture}[thick,scale=0.5,rotate=-90,transform shape]
\filldraw[fill=black!50!white, draw=black] (1,0) -- (0,0) -- (0,1)--(1,1)--(1,0)--(0,0);
\filldraw[fill=black!50!white, draw=black] (1,1) -- (0,1) -- (0,2)--(1,2)--(1,1)--(0,1);
\filldraw[fill=black!20!white, draw=black] (2,0) -- (1,0) -- (1,1)--(2,1)--(2,0)--(1,0);
\draw(2,1) -- (1,1) -- (1,2)--(2,2)--(2,1)--(1,1);
\end{tikzpicture}};
\node (Q1) [below left of=Q01] {\begin{tikzpicture}[thick,scale=0.5,rotate=-90,transform shape]
\filldraw[fill=black!50!white, draw=black] (1,0) -- (0,0) -- (0,1)--(1,1)--(1,0)--(0,0);
\filldraw[fill=black!50!white, draw=black] (1,1) -- (0,1) -- (0,2)--(1,2)--(1,1)--(0,1);
\filldraw[fill=black!20!white, draw=black] (2,0) -- (1,0) -- (1,1)--(2,1)--(2,0)--(1,0);
\filldraw[fill=black!20!white, draw=black](2,1) -- (1,1) -- (1,2)--(2,2)--(2,1)--(1,1);
\end{tikzpicture}};
\draw[->, thick] (Qc) -- (Qf);
\draw[->, thick] (Q12) -- (Qc);
\draw[->, thick] (Q21) -- (Qc);
\draw[->, thick] (Q01) -- (Q21);
\draw[->, thick] (Q10) -- (Q21);
\draw[->, thick] (Q01) -- (Q12);
\draw[->, thick] (Q10) -- (Q12);
\draw[->, thick] (Q1) -- (Q01);
\draw[->, thick] (Q1) -- (Q10);
\end{tikzpicture}
\end{equation}
This gives an approximate picture of the Jacopson topology on the set of irreducible representations. It is not hard to see that for any irreducible representation $\pi$ belonging to one of the families in~\eqref{reps}, there exists irreducible representations from the other families such that $\pi$ can be fitted into a diagram like~\eqref{jacob}.

\section{Proof of the main result}
We will first show that the closure of $\mathrm{Pol}(\mathrm{Mat}_{2})_{q}$ under the coherent representation $\Omega_{\varphi}$ does not depend on $q.$ It is easy to the from the uniqueness of the coherent representation that
\begin{equation}\label{coriso}
\Omega_{\varphi}\cong \Omega_{0}\circ \Psi_{\varphi,0},
\end{equation}
so these $C^{*}$-algebras does not depend on the parameter $\varphi\in [0,2\pi).$ Hence, we need only to prove the result for
$$
B_{q}:=\overline{\Omega_{0}(\mathrm{Pol}(\mathrm{Mat}_{2})_{q})}
$$
Let \begin{equation}\label{x}x:=(I-z_{2}^{1}(z_{2}^{1})^{*}-z_{2}^{2}(z_{2}^{2})^{*})(I-z_{1}^{2}(z_{1}^{2})^{*}-z_{2}^{2}(z_{2}^{2})^{*}).\end{equation} Notice that we have the invariance
\begin{equation}\label{xinv}
\begin{array}{ccc}
\Psi_{\varphi_{1},\varphi_{2}}(x)=x, & \text{(for all $\varphi_{1},\varphi_{2}\in [0,2\pi)$)}.
\end{array}
\end{equation}
 We can calculate
$$
\Omega_{0}((I-z_{2}^{1}(z_{2}^{1})^{*}-z_{2}^{2}(z_{2}^{2})^{*}))=d_{q}^{2}\otimes d_{q}^{2}\otimes I
$$
$$
\Omega_{0}((I-z_{2}^{1}(z_{2}^{1})^{*}-z_{2}^{2}(z_{2}^{2})^{*}))=I\otimes d_{q}^{2}\otimes d_{q}^{2},
$$
where $d_{q}$ is given by~\eqref{Cd}. It follows that
\begin{equation}\label{omegax}
\Omega_{0}(x)= d_{q}^{2}\otimes d_{q}^{4}\otimes d_{q}^{2}.
\end{equation}
If we let $L\subseteq \mathrm{Pol}(\mathrm{Mat}_{2})_{q}$ be the ideal generated by $x,$ then as $\Omega_{0}$ is irreducible, we have 
$$
\overline{\Omega_{0}(L)}= \K\otimes \K\otimes \K
$$
Let $\h:=\ell^{2}(\mathbb{Z}_{+})^{\otimes 3}.$ We denote the compact operators on $\h$ by $\K$ (abusing the notation). It follows that $B_{q}/\K$ is a sub-algebra of $\mathcal Q(\h).$ Consider now the composition
$$
\Gamma_{q}:\mathrm{Pol}(\mathrm{Mat}_{2})_{q}\overset{\Omega_{0}}{\to} B_{q}\to B_{q}+\K.
$$ 
By~\eqref{coriso} and~\eqref{xinv}, it follows that $\Omega_{\varphi}(x)\neq 0$ for all $\varphi\in [0,2\pi),$ and since the Fock representation is faithful, also $\pi_{F}(x)\neq 0.$ But as $\Gamma_{q}(x)=0,$ it must be the case that neither the Fock representation, nor any of the coherent representations can be sub-summands of $\Gamma_{q}.$

It follows that only the $5$ families of irreducible representations

\begin{equation}\label{repsred}
\begin{array}{ccccc}
\begin{tikzpicture}[thick,scale=0.5]
\filldraw[fill=black!20!white, draw=black](1,0) -- (0,0) -- (0,1)--(1,1)--(1,0)--(0,0);
\filldraw[fill=black!50!white, draw=black]  (1,1) -- (0,1) -- (0,2)--(1,2)--(1,1)--(0,1);
\draw (2,0) -- (1,0) -- (1,1)--(2,1)--(2,0)--(1,0);
\draw (2,1) -- (1,1) -- (1,2)--(2,2)--(2,1)--(1,1);
\node at (0.5,-0.5) {$1$};
\node at (1.5,-0.5) {$2$};
\node at (2.5,0.5) {$2$};
\node at (2.5,1.5) {$1$};
\end{tikzpicture}
&
\begin{tikzpicture}[thick,scale=0.5]
\draw (1,0) -- (0,0) -- (0,1)--(1,1)--(1,0)--(0,0);
\filldraw[fill=black!50!white, draw=black] (1,1) -- (0,1) -- (0,2)--(1,2)--(1,1)--(0,1);
\draw(2,0) -- (1,0) -- (1,1)--(2,1)--(2,0)--(1,0);
\filldraw[fill=black!20!white, draw=black] (2,1) -- (1,1) -- (1,2)--(2,2)--(2,1)--(1,1);
\node at (0.5,-0.5) {$1$};
\node at (1.5,-0.5) {$2$};
\node at (2.5,0.5) {$2$};
\node at (2.5,1.5) {$1$};
\end{tikzpicture}
&
\begin{tikzpicture}[thick,scale=0.5]
\filldraw[fill=black!20!white, draw=black] (1,0) -- (0,0) -- (0,1)--(1,1)--(1,0)--(0,0);
\filldraw[fill=black!50!white, draw=black] (1,1) -- (0,1) -- (0,2)--(1,2)--(1,1)--(0,1);
\draw (2,0) -- (1,0) -- (1,1)--(2,1)--(2,0)--(1,0);
\filldraw[fill=black!20!white, draw=black](2,1) -- (1,1) -- (1,2)--(2,2)--(2,1)--(1,1);
\node at (0.5,-0.5) {$1$};
\node at (1.5,-0.5) {$2$};
\node at (2.5,0.5) {$2$};
\node at (2.5,1.5) {$1$};
\end{tikzpicture}
\end{array}
$$
$$
\begin{array}{ccc}
\begin{tikzpicture}[thick,scale=0.5]
\filldraw[fill=black!50!white, draw=black] (1,0) -- (0,0) -- (0,1)--(1,1)--(1,0)--(0,0);
\filldraw[fill=black!50!white, draw=black] (1,1) -- (0,1) -- (0,2)--(1,2)--(1,1)--(0,1);
\filldraw[fill=black!20!white, draw=black] (2,0) -- (1,0) -- (1,1)--(2,1)--(2,0)--(1,0);
\draw(2,1) -- (1,1) -- (1,2)--(2,2)--(2,1)--(1,1);
\node at (0.5,-0.5) {$1$};
\node at (1.5,-0.5) {$2$};
\node at (2.5,0.5) {$2$};
\node at (2.5,1.5) {$1$};
\end{tikzpicture}
&
\begin{tikzpicture}[thick,scale=0.5]
\filldraw[fill=black!50!white, draw=black] (1,0) -- (0,0) -- (0,1)--(1,1)--(1,0)--(0,0);
\filldraw[fill=black!50!white, draw=black] (1,1) -- (0,1) -- (0,2)--(1,2)--(1,1)--(0,1);
\filldraw[fill=black!20!white, draw=black] (2,0) -- (1,0) -- (1,1)--(2,1)--(2,0)--(1,0);
\filldraw[fill=black!20!white, draw=black](2,1) -- (1,1) -- (1,2)--(2,2)--(2,1)--(1,1);
\node at (0.5,-0.5) {$1$};
\node at (1.5,-0.5) {$2$};
\node at (2.5,0.5) {$2$};
\node at (2.5,1.5) {$1$};
\end{tikzpicture}
\end{array}
\end{equation}
can be found as sub-summands of $\Gamma_{q}.$ 
\\

Let $\Xi_{q},\Phi_{q}$ be the representations we get by taking the direct integral, with respect to $\varphi,$ of the irreducible families
\begin{equation}\label{repxi}
\begin{array}{ccc}
\begin{tikzpicture}[baseline=12,thick,scale=0.5]
\filldraw[fill=black!20!white, draw=black](1,0) -- (0,0) -- (0,1)--(1,1)--(1,0)--(0,0);
\filldraw[fill=black!50!white, draw=black]  (1,1) -- (0,1) -- (0,2)--(1,2)--(1,1)--(0,1);
\draw (2,0) -- (1,0) -- (1,1)--(2,1)--(2,0)--(1,0);
\draw (2,1) -- (1,1) -- (1,2)--(2,2)--(2,1)--(1,1);
\node at (0.5,-0.5) {$1$};
\node at (1.5,-0.5) {$2$};
\node at (2.5,0.5) {$2$};
\node at (2.5,1.5) {$1$};
\end{tikzpicture}
& \text{(to get $\Xi_{q}$)}
\end{array}
\end{equation}
\begin{equation}\label{repphi}
\begin{array}{cc}
\begin{tikzpicture}[baseline=12,thick,scale=0.5]
\draw (1,0) -- (0,0) -- (0,1)--(1,1)--(1,0)--(0,0);
\filldraw[fill=black!50!white, draw=black] (1,1) -- (0,1) -- (0,2)--(1,2)--(1,1)--(0,1);
\draw(2,0) -- (1,0) -- (1,1)--(2,1)--(2,0)--(1,0);
\filldraw[fill=black!20!white, draw=black] (2,1) -- (1,1) -- (1,2)--(2,2)--(2,1)--(1,1);
\node at (0.5,-0.5) {$1$};
\node at (1.5,-0.5) {$2$};
\node at (2.5,0.5) {$2$};
\node at (2.5,1.5) {$1$};
\end{tikzpicture}
& \text{(to get $\Phi_{q}$)}
\end{array}
\end{equation}
so that 
$$
\Xi_{q}:\mathrm{Pol}(\mathrm{Mat}_{2})_{q}\to C(\mathbb{T})\otimes C^{*}(S)\otimes C^{*}(S)\subseteq\mathcal B(\h_{1}),
$$
$$
\Phi_{q}:\mathrm{Pol}(\mathrm{Mat}_{2})_{q}\to C^{*}(S)\otimes C^{*}(S)\otimes C(\mathbb{T})\subseteq\mathcal B(\h_{2}),
$$
where $\h_{1}=L^{2}(\mathbb{T})\otimes \ell^{2}(\mathbb{Z}_{+})\otimes \ell^{2}(\mathbb{Z}_{+})$ and $\h_{2}=\ell^{2}(\mathbb{Z}_{+})\otimes \ell^{2}(\mathbb{Z}_{+})\otimes L^{2}(\mathbb{T}).$
Notice that the representations in~\eqref{repsred} are all quotients of either~\eqref{repxi} or~\eqref{repphi}. Hence, as for any $a\in\mathrm{Pol}(\mathrm{Mat}_{2})_{q}$ the norm of $(\Xi_{q}\oplus \Phi_{q})(a)$ is the supremum of the norms of $a$ under the families of $*$-representations~\eqref{repxi} and~\eqref{repphi}, and it follows that we have the inequality
\begin{equation}\label{norms}
\begin{array}{cc}
||\Gamma_{q}(a)||\leq ||(\Xi_{q}\oplus \Phi_{q})(a)||, & \text{for all $a\in\mathrm{Pol}(\mathrm{Mat}_{2})_{q}.$}
\end{array}
\end{equation}
It is not hard to verify that the two $C^{*}$-algebras $$\overline{\Xi_{q}(\mathrm{Pol}(\mathrm{Mat}_{2})_{q})}\subseteq\mathcal B(\h_{1}),$$ $$\overline{\Phi_{q}(\mathrm{Pol}(\mathrm{Mat}_{2})_{q})}\subseteq\mathcal B(\h_{2})$$ are both independent of $q$, but we omit the proof here. We prove instead:
\begin{prop}\label{prop}
The $C^{*}$-algebra 
 $$E=\overline{(\Xi_{q}\oplus \Phi_{q})(\mathrm{Pol}(\mathrm{Mat}_{2})_{q}})$$ does not depend on $q.$ Moreover, we have a family of injective $*$-homomorphisms $$\phi_{q}:E\to B_{q}/\K\subseteq\mathcal Q(\h)$$ varying point-norm continuously on $q\in (0,1).$ Hence, by $\fact$, the $C^{*}$-algebras $B_{q}$ are isomorphic for all $q.$
\end{prop}
\begin{proof}
Essentially, the proof is to show that we can let $q\to 0$ on the generators and still get the same $C^{*}$-algebra. If we calculate the image of the generators $\{z_{m}^{j}\}_{m,j=1}^{2}\subseteq \mathrm{Pol}(\mathrm{Mat}_{2})_{q}$ under $\Xi_{q}\oplus \Phi_{q},$ we get 

\begin{gather}
(\Xi_{q}\oplus \Phi_{q})(z_{2}^{2})=(I\otimes C_{q}S\otimes I)\oplus (I\otimes C_{q}S\otimes I)\label{im22} \\ (\Xi_{q}\oplus \Phi_{q})(z_{1}^{2})=(z\otimes d_{q}\otimes I)\oplus (C_{q}S\otimes d_{q}\otimes I)\label{im12}\\
(\Xi_{q}\oplus \Phi_{q})(z_{2}^{1})=(I\otimes d_{q}\otimes C_{q}S)\oplus (I\otimes d_{q}\otimes z)\label{im21} \\ (\Xi_{q}\oplus \Phi_{q})(z_{1}^{1})=-q^{-1}(z\otimes S^{*}C_{q}\otimes C_{q}S)\oplus (C_{q}S\otimes S^{*}C_{q}\otimes z). \label{im11}
\end{gather}

If we multiply~\eqref{im11} by $-q,$ then all the images above has well-defined limits as $q\to 0.$ The limits of~\eqref{im22}-\eqref{im11} are respectively

\begin{gather}
Z_{2}^{2}:=(I\otimes S\otimes I)\oplus (I\otimes S\otimes I)\label{lim22}\\
Z_{1}^{2}:=(z\otimes P\otimes I)\oplus (S\otimes P\otimes I)\label{lim12}\\
Z_{2}^{1}:=(I\otimes P\otimes S)\oplus (I\otimes P\otimes z)\label{lim21}\\
Z_{1}^{1}:=(z\otimes S^{*}\otimes S)\oplus (S\otimes S^{*}\otimes z).\label{lim11}
\end{gather}
It is easy to see that the right-hand sides of~\eqref{lim22}-\eqref{lim21} are in the $C^{*}$-algebra generated by~\eqref{im22}-\eqref{im21} and vice versa. This leaves us to show the similar statement for~\eqref{im11} and~\eqref{lim11}. However, it is easy to see that $Z_{1}^{1}$ is in the $C^{*}$-algebra generated by the images of $z_{1}^{1}$ and $z_{2}^{2}.$ To show that we can get the right-hand side of~\eqref{im11} from~\eqref{lim22}-\eqref{lim11} we first notice that a left-inverse of $(I\otimes C_{q}S\otimes I)\oplus (I\otimes C_{q}S\otimes I)$ is in the $C^{*}$-algebra generated by~\eqref{lim22} and if we multiply this operator on the right of the right-hand side of~\eqref{im11} we see that we only need to prove that the operator 
$$
(z\otimes I\otimes C_{q}S)\oplus (C_{q}S\otimes I\otimes z)
$$
is in the $C^{*}-$algebra generated by~\eqref{lim22}-\eqref{lim11}. Actually, we only need to prove that
$$
(I\otimes I\otimes C_{q}^{2})\oplus (C_{q}^{2}\otimes I\otimes I)=(I\otimes I\otimes (I-d_{q}^{2}))\oplus ((I-d_{q}^{2})\otimes I\otimes I)=
$$
$$
(I\otimes I\otimes I)\oplus (I\otimes I\otimes I)-(I\otimes I\otimes d_{q}^{2})\oplus (d_{q}^{2}\otimes I\otimes I)
$$
is in this $C^{*}$-algebra. If we consider $Z_{1}^{1}Z_{2}^{2}=(z\otimes I\otimes S)\oplus (S\otimes I\otimes z),$ then 
$$
I-(Z_{1}^{1}Z_{2}^{2})(Z_{1}^{1}Z_{2}^{2})^{*}=(I\otimes I\otimes P)\oplus (P\otimes I\otimes I)
$$
and thus
$$
 (I\otimes I\otimes d_{q}^{2})\oplus(d_{q}^{2}\otimes I\otimes I)=\sum_{k=0}^{\infty}q^{2k}(Z_{1}^{1}Z_{2}^{2})^{k}(I-(Z_{1}^{1}Z_{2}^{2})(Z_{1}^{1}Z_{2}^{2})^{*})(Z_{1}^{1}Z_{2}^{2})^{*k}
$$
and this gives the first claim.
\\

Let us prove the existence and injectivity of the maps $\phi_{q}:E\to\mathcal Q(\h).$ We do this by establishing a isomorphism $E\cong B_{q}/\K.$ Firstly, we do have a homomorphism $B_{q}\to E$ as it is easy to see from~\eqref{reps} that both $\Xi_{q}$ and $\Phi_{q}$ can be constructed from $\Omega_{q}$ by taking the quotients of certain ideals: in the case of $\Xi_{q},$ we take quotient of $B_{q}$ with the ideal $J_{\Phi}:=B_{q}\cap (\K\otimes C^*(S)\otimes C^*(S))$ and in the case of $\Phi_{q}$ we use the ideal $J_{\Xi}:=B_{q}\cap (C^*(S)\otimes C^*(S)\otimes \K),$ so that $\Xi_{q}$ and $\Phi_{q}$ respectively factors as
$$
\mathrm{Pol}(\mathrm{Mat}_{2})_{q}\overset{\Omega_{0}}{\to } B_{q}\overset{p_{\Xi}}{\to} B_{q}/J_{\Xi} 
$$
$$
\mathrm{Pol}(\mathrm{Mat}_{2})_{q}\overset{\Omega_{0}}{\to } B_{q}\overset{p_{\Phi}}{\to}  B_{q}/J_{\Phi} 
$$
where $p_{\Xi}$ and $p_{\Phi}$ are the quotient maps. It thus follows that we have a surjective homomorphism 
$$
p_{\Xi}\oplus p_{\Phi}:B_{q}\to E
$$
such that $\Xi_{q}\oplus \Phi_{q}=(p_{\Xi}\oplus p_{\Phi})\circ \Omega_{0}.$ Moreover, we have $\K\subseteq J_{\Xi}\cap J_{\Phi}$ and hence there are surjective homomorphisms
$$
\begin{array}{cc}
\psi_{q}: B_{q}/\K\to E, & q\in (0,1).
\end{array}
$$
such that $\Xi_{q}\oplus \Phi_{q}$ factors as 
$$
\mathrm{Pol}(\mathrm{Mat}_{2})_{q}\overset{\Omega_{0}}{\to } B_{q}\to B_{q}/\K \overset{\psi_{q}}{\to} E
$$
i.e $\Xi_{q}\oplus \Phi_{q}=\psi_{q}\circ \Gamma_{q}.$ We claim that $\psi_{q}$ is actually isometric. To show this, we only need to prove that the map 
\begin{equation}\label{psiq} \begin{array}{cc} \Gamma_{q}(a)\mapsto (\Xi_{q}\oplus \Phi_{q})(a), & a\in\mathrm{Pol}(\mathrm{Mat}_{2})_{q}\end{array}\end{equation} is an isometry. This map is certainly a contraction, as $\psi_{q}$ is. The opposite inequality follows from~\eqref{norms}. Now let $\phi_{q}:=\psi_{q}^{-1}.$
\\

We must show that the homomorphisms $\phi_{q},$ $q\in(0,1)$ are varying norm-continuously on $q,$ i.e that for any fixed element $a\in E$ and $q\in (0,1)$ and $\epsilon>0,$ there is a $\delta>0$ such that
$$
\begin{array}{cc}
||\phi_{q}(a)-\phi_{s}(a)||<\epsilon, & \text{for $s\in(0,1)$ such that $|q-s|<\delta.$}
\end{array}
$$
We can assume that $a\in E$ is in the image of $\mathrm{Pol}(\mathrm{Mat}_{2})_{q}.$ Consider now a polynomial function $$J(x_{1},x_{2},x_{3},x_{4},x_{1}^{*},x_{2}^{*},x_{3}^{*},x_{4}^{*})$$ in the non-commutative variables $\{x_{1},x_{2},x_{3},x_{4},x_{1}^{*},x_{2}^{*},x_{3}^{*},x_{4}^{*}\},$ let
 \begin{equation}\label{jpol}
J(Z_{q}):=J((q z_{1}^{1}),z_{2}^{1},z_{1}^{2},z_{2}^{2},(q z_{1}^{1})^{*},(z_{2}^{1})^{*},(z_{1}^{2})^{*},(z_{2}^{2})^{*}).
\end{equation}
There exists such polynomial $J$ such that $(\Xi_{q}\oplus \Phi_{q})(J(Z_{q}))=a.$ To emphasize the dependence on $q\in(0,1),$ let us write $\Omega_{0}^{(q)}$ for the coherent representation of $\mathrm{Pol}(\mathrm{Mat}_{2})_{q}.$ It is easy to see from the actions of the representations $\Omega_{0}^{(q)}$ and $\Xi_{q}\oplus \Phi_{q}$ that there is a $\delta>0$ such that 
$$
\begin{array}{ccc}
||\Omega_{0}^{(s)}(J(Z_{s}))-\Omega_{0}^{(q)}(J(Z_{q}))||<\frac{\epsilon}{2},\\ ||(\Xi_{s}\oplus \Phi_{s})(J(Z_{s}))-(\Xi_{q}\oplus \Phi_{q})(J(Z_{q}))||<\frac{\epsilon}{2},&
\end{array}
$$ 
for $|q-s|<\delta$, and thus, since $\phi_{s}$ are contractions and $$\phi_{s}((\Xi_{s}\oplus \Phi_{s})(J(Z_{s})))=\Omega_{0}^{(s)}(J(Z_{s}))+\K,$$ we get for $|q-s|<\delta$
$$
||\phi_{s}(a)-\phi_{q}(a)||=
$$
$$
=||\Omega_{0}^{(s)}(J(Z_{s}))-\Omega_{0}^{(q)}(J(Z_{q}))+\phi_{s}((\Xi_{q}\oplus \Phi_{q})(J(Z_{q}))-(\Xi_{s}\oplus \Phi_{s})(J(Z_{s})))+\K||\leq
$$
$$
\leq ||\Omega_{0}^{(s)}(J(Z_{s}))-\Omega_{0}^{(q)}(J(Z_{q}))||+||(\Xi_{q}\oplus \Phi_{q})(J(Z_{q}))-(\Xi_{s}\oplus \Phi_{s})(J(Z_{s}))||<\frac{\epsilon}{2}+\frac{\epsilon}{2}.
$$
We can now apply the $\fact$ to get that the $C^{*}$-algebras $B_{q}/\K+\K =B_{q}$ are isomorphic for all $q\in (0,1).$
\end{proof}
\begin{cor}
for every $q\in (0,1)$ the $C^{*}$-algebra $B_{q}$ is isomorphic to the $C^{*}$-algebra $B_{0}\subseteq C^{*}(S)\otimes C^*(S)\otimes C^*(S)$ generated by the operators 
\begin{equation}\label{b0gen}
\begin{array}{ccc}
S\otimes S^{*}\otimes S, & I\otimes P\otimes S,\\
S\otimes P\otimes I, & I\otimes S\otimes I.
\end{array}
\end{equation}
\end{cor}
\begin{proof}
We argue first that the limit $\phi_{0}:=\lim_{q\to 0}\phi_{q}:E\to\mathcal Q(\h)$ exists in the topology of point-wise convergence. For every non-commutative polynomial $J$ like~\eqref{jpol}, both $\Omega_{0}^{(q)}(J(Z_{q}))$ and $(\Xi_{q}\oplus \Phi_{q})(J(Z_{q}))$ converges to limits as $q\to 0$ (we write again $\Omega_{0}^{(q)}$ to emphasize dependence on $q$). In fact, this follows from that we have convergence
\begin{equation}\label{limb0}
\begin{array}{ccc}
\lim_{q\to 0}\Omega_{0}^{(q)}(qz_{1}^{1})=S\otimes S^*\otimes S, & \lim_{q\to 0}\Omega_{0}^{(q)}(z_{2}^{1})=S\otimes P\otimes I,\\
 \lim_{q\to 0}\Omega_{0}^{(q)}(z_{1}^{2})=I\otimes P\otimes S, &  \lim_{q\to 0}\Omega_{0}^{(q)}(z_{2}^{2})=I\otimes S\otimes I
\end{array}
\end{equation}
as well as the limits from~\eqref{lim22}-\eqref{lim11}, and this shows the convergence of $\lim_{q\to 0}\phi_{q}((\Xi_{q}\oplus \Phi_{q})(J(Z_{q})))=\lim_{q\to 0}\Omega_{0}^{(q)}(J(Z_{q}))+\K$ for such $J.$ As the elements~\eqref{lim22}-\eqref{lim11} generates $E,$ there is for any $x\in E$ a non-commutative polynomial $J,$ such that
$$
\lim_{q\to 0}||x-(\Xi_{q}\oplus\Phi_{q})(J(Z_{q}))||<\frac{\epsilon}{2}
$$ 
and hence, we have for $q,s\in (0,1)$ small enough
$$
||\phi_{q}(x)-\phi_{s}(x)||\leq
$$

$$
\leq||\phi_{q}(x-(\Xi_{q}\oplus \Phi_{q})(J(Z_{q})))||+ ||\phi_{s}(x-(\Xi_{s}\oplus \Phi_{s})(J(Z_{s})))||+
$$

$$
+||\phi_{q}\circ(\Xi_{q}\oplus \Phi_{q})(J(Z_{q}))-\phi_{s}\circ(\Xi_{s}\oplus \Phi_{s})(J(Z_{s}))||=
$$

$$
=||\phi_{q}(x-(\Xi_{q}\oplus \Phi_{q})(J(Z_{q})))||+||\phi_{s}(x-(\Xi_{s}\oplus \Phi_{s})(J(Z_{s})))|| +
$$

$$
+||\Omega_{0}^{(q)}(J(Z_{q}))-\Omega_{0}^{(s)}(J(Z_{s}))+\K||\leq
$$
$$
\leq||x-(\Xi_{q}\oplus \Phi_{q})(J(Z_{q}))||+||x-(\Xi_{s}\oplus \Phi_{s})(J(Z_{s}))||+
$$
$$
+||(\Omega^{(q)}_{0}(J(Z_{q}))-\Omega_{0}^{(s)}(J(Z_{s}))||\leq \epsilon.
$$
Thus $\phi_{0}:E\to \mathcal Q(\h)$ exists and maps 
\begin{equation}\label{phi0}
\begin{array}{ccc}
\phi_{0}(Z_{1}^{1})=S\otimes S^{*}\otimes S+\K, & \phi_{0}(Z_{1}^{2})=S\otimes P\otimes I+\K\\
 \phi_{0}(Z_{1}^{2})=I\otimes P\otimes S+\K, &  \phi_{0}(Z_{2}^{2})=I\otimes S\otimes I+\K.
\end{array}
\end{equation}
with $Z_{i}^{j}$ as~\eqref{lim11}-\eqref{lim22}. Moreover, the point-wise limit of injective (hence isometric) $*$-homomorphisms must be still injective, as $||\phi_{0}(x)||=\lim_{q\to 0}||\phi_{q}(x)||=\lim_{q\to 0}||x||=||x||$ for all $x\in E.$ It is easy to see that $\K\subseteq B_{0}$ and thus from~\eqref{phi0} it follows that $\phi_{0}(E)+\K=B_{0}.$ The corollary now follows from $\fact$.
\end{proof}
It is not hard to see that the $*$-representation $\Pi_{q}$ from {\bf\textsc{Example}} is equivalent to the representation determined as 
\begin{equation}\label{piq}
\begin{array}{ccc}
z_{m}^{j}\mapsto z\otimes \Omega_{0}(z_{m}^{j})\in C(\mathbb{T})\otimes B_{q}\subseteq C(\mathbb{T})\otimes C^{*}(S)^{\otimes 3}, & \text{(for $j,m=1,2$)}
\end{array}
\end{equation}
and henceforth this will be our definition of $\Pi_{q}.$ It thus follows that $A_{q},$ defined as the closure of the range of $\Pi_{q},$ is a sub-algebra of $C(\mathbb{T})\otimes B_{q}.$
\begin{lem}\label{inclusion}
For every $q\in (0,1),$ we have the inclusions
\begin{equation}\label{inc}
C(\mathbb{T})\otimes \K\subseteq A_{q}\subseteq C(\mathbb{T})\otimes B_{q}
\end{equation}
\end{lem} 
\begin{proof}
We only need to prove that \begin{equation}\label{subset}I\otimes \K\subseteq A_{q}.\end{equation} The full claim of~\eqref{inc} then follow from this, since as $\Pi_{q}(z_{2}^{2})=z\otimes I\otimes C_{q}S\otimes I,$ we have 
$$z^{k}\otimes \K=
\begin{cases}
(I\otimes \K)\cdot\Pi_{q}(z_{2}^{2})^{k}, & \text{(for $k\geq 0$),}\\
\Pi_{q}(z_{2}^{2})^{*k}\cdot(I\otimes \K), & \text{(for $k< 0$),} \end{cases}$$
and hence $C(\mathbb{T})\otimes \K\subseteq A_{q}$ by closure. To prove~\eqref{subset}, we first prove $I\otimes P \otimes P\otimes P\in A_{q}.$ Recall the element $x\in \mathrm{Pol}(\mathrm{Mat}_{2})_{q}$ defined by~\eqref{x}. Similarly as to~\eqref{omegax}, we get 
$$
\Pi_{q}(x)=I\otimes d_{q}^{2}\otimes d_{q}^{4}\otimes d_{q}^{2} \in A_{q},
$$
and an application of the continuous functional calculus gives that $I\otimes P\otimes P\otimes P\in A_{q}.$
Recall that for the coherent representation we have $\Omega_{\varphi}(z_{1}^{1})(e_{0}\otimes e_{0}\otimes e_{0})=e^{i\varphi}(e_{0}\otimes e_{0}\otimes e_{0})$ and $\Omega_{\varphi}(z_{1}^{1})^{*}(e_{0}\otimes e_{0}\otimes e_{0})=e^{-i\varphi}(e_{0}\otimes e_{0}\otimes e_{0}).$ It follows that for any $f(z)\in L^{2}(\mathbb{T}),$ we have $$\Pi_{q}(z_{1}^{1})^{k}(f(z)\otimes e_{0}\otimes e_{0}\otimes e_{0})=z^{k}f(z)\otimes e_{0}\otimes e_{0}\otimes e_{0}$$
$$\Pi_{q}(z_{1}^{1})^{*k}(f(z)\otimes e_{0}\otimes e_{0}\otimes e_{0})=z^{-k}f(z)\otimes e_{0}\otimes e_{0}\otimes e_{0}.$$
Moreover, recall that for any basis vector $e_{k}\otimes e_{j}\otimes e_{m}\in \ell^{2}(\mathbb{Z}_{+})^{\otimes 3}$ there is a monomial $M_{k,j,m}$ in $z_{1}^{2},z_{2}^{1},z_{2}^{2}$ and of degree $k+j+m$ such that $$\Omega_{0}(M_{k,j,m})(e_{0}\otimes e_{0}\otimes e_{0})=e_{k}\otimes e_{j}\otimes e_{m}.$$ It follows that for any $f(z)\in L^{2}(\mathbb{T}) ,$ we have 
$$
\Pi_{q}((M_{k,j,m})(z_{1}^{1})^{*(k+j+m)})(f(z)\otimes e_{0}\otimes e_{0}\otimes e_{0})=f(z)\otimes e_{k}\otimes e_{j}\otimes e_{m}
$$
and hence $\Pi_{q}((M_{k,j,m})(z_{1}^{1})^{*(k+j+m)})(I\otimes P\otimes P\otimes P)\in A_{q}$ is the operator
$$\begin{array}{ccc}
f(z)\otimes x\mapsto \langle x, e_{0}\otimes e_{0}\otimes e_{0}\rangle (f(z)\otimes e_{k}\otimes e_{j}\otimes e_{m}), & \text{(for $f(z)\in L^{2}(\mathbb{T})$ and $x\in\ell^{2}(\mathbb{Z}_{+})^{\otimes 3}$)}\end{array}
$$
and these operators generates $I\otimes \K$ as a $C^{*}$-algebra.
\end{proof}
With this result in hand, we can now prove that $A_{q}$ is independent of $q.$
\begin{thm}
Let $\Psi_{q}:B_{0}\to B_{q}$ be the continuous field of $*$-isomorphisms acquired from $\fact$. For every $q,s\in(0,1),$ the homomorphism $\id\otimes (\Psi_{q}\circ \Psi_{s}^{-1})$ is an isomorphism $A_{s}\to A_{q}.$
\end{thm}
\begin{proof}
Clearly, we only need to prove the inclusion 
$$
\left(\id\otimes (\Psi_{q}\circ \Psi_{s}^{-1})\right)(A_{s})\subseteq A_{q}.
$$
By $\lemma$~\ref{inclusion} $C(\mathbb{T})\otimes \K\subseteq A_{q}$ for every $q\in (0,1)$ and hence we have the quotient homomorphism (if we again let $\h:=\ell^{2}(\mathbb{Z}_{+})^{\otimes 3}$)
$$
(\id\otimes p):C(\mathbb{T})\otimes\mathcal B(\h)\to C(\mathbb{T})\otimes\mathcal B(\h)/(C(\mathbb{T})\otimes \K)= C(\mathbb{T})\otimes \mathcal Q(\h)
$$
such that 
\begin{equation}\label{qomega}
\begin{array}{ccc}
(\id\otimes p)\circ \Pi_{q}(z_{m}^{j})=z\otimes \left(\Omega^{(q)}_{0}(z_{m}^{j})+\K\right)= z\otimes \phi_{q}((\Xi_{q}\oplus \Phi_{q})(z_{m}^{j})), & \text{(for $m,j=1,2$)}
\end{array}
\end{equation}
by the properties of the $*$-homomorphisms $\phi_{q}.$ Consider the $*$-representation of $\mathrm{Pol}(\mathrm{Mat}_{2})_{q}$ determined by 
$$
\begin{array}{ccc}
z_{m}^{j}\mapsto z\otimes (\Xi_{q}\oplus \Phi_{q})(z_{m}^{j}), & \text{for $m,j=1,2$}
\end{array}
$$
and consider the closure of the image of $\mathrm{Pol}(\mathrm{Mat}_{2})_{q}$ in $C(\mathbb{T})\otimes E.$ We can use the exact same arguments as in $\proposition$~\ref{prop} to show that the $C^{*}$-sub-algebra of $C(\mathbb{T})\otimes E$ thus acquired does not depend on $q\in (0,1).$ Let us denote this $C^*$-algebra by $F.$ It now follows from~\eqref{qomega} that 
\begin{equation}\label{quotient}
\begin{array}{ccc}
(\id\otimes \phi_{q})(F)=(\id\otimes p)(A_{q}), & \text{(for all $q\in(0,1)$).}
\end{array}
\end{equation}
As $\fact$ implies that the following diagram commutes
\begin{equation}\label{com}
\begin{xy}\xymatrix{
B_{q}\ar[r]^*{\Psi_{s}\circ \Psi_{q}^{-1}}\ar[d]_*{p}& B_{s}\ar[d]^*{p}\\
\phi_{q}(E)\ar[r]_*{\phi_{s}\circ \phi_{q}^{-1}} & \phi_{s}(E)
}\end{xy}
\end{equation}

it then follows that 
$$
(\id\otimes p)((\id\otimes (\Psi_{s}\circ\Psi_{q}^{-1}))(A_{q}))=(\id\otimes (p\circ\Psi_{s}\circ\Psi_{q}^{-1}))(A_{q}))=
$$
$$
=(\id\otimes \phi_{s}\circ \phi_{q}^{-1})((\id\otimes p)(A_{q}))=(\id\otimes \phi_{s}\circ \phi_{q}^{-1})((\id\otimes \phi_{q})(F))=(\id\otimes \phi_{s})(F).
$$
By $\lemma$~\ref{inclusion} and~\eqref{quotient}, we have
$$
(\id\otimes p)^{-1}((\id\otimes \phi_{q})(F))=A_{q}
$$
and this gives the theorem.
\end{proof}

\end{document}